\documentclass[11pt]{article}
\usepackage{lmodern}
\usepackage[T1]{fontenc}
\usepackage{textcomp}
\usepackage[utf8]{inputenc}
\usepackage{color}
\usepackage{mathtools}
\usepackage{enumitem}
\usepackage{amsmath}
\usepackage{amsthm}
\usepackage{amssymb}
\usepackage{geometry}
\usepackage[authoryear,round]{natbib}
\usepackage{microtype}
\usepackage[bookmarks=false,
 breaklinks=false,pdfborder={0 0 1},backref=section,colorlinks=true]
 {hyperref}
\hypersetup{
 linkcolor=blue,citecolor=blue,urlcolor=blue}

\makeatletter
\@ifundefined{date}{}{\date{}}
\usepackage{bm}\usepackage{booktabs}

\newtheorem{theorem}{Theorem}[section]
\newtheorem{proposition}[theorem]{Proposition}\newtheorem{lemma}[theorem]{Lemma}\newtheorem{corollary}[theorem]{Corollary}\theoremstyle{definition}
\newtheorem{definition}[theorem]{Definition}\newtheorem{assumption}[theorem]{Assumption}\theoremstyle{remark}
\newtheorem{remark}[theorem]{Remark}

\newcommand{\tr}{\mathop{\mathrm{tr}}}
\newcommand{\diag}{\mathop{\mathrm{diag}}}
\newcommand{\argmax}{\mathop{\mathrm{arg\,max}}}

\newcommand{\dist}{\mathop{\mathrm{dist}}}

\newcommand{\KL}{\mathop{\mathrm{KL}}}

\title{Consistency of the Bayesian Information Criterion for Model Selection in Exploratory Factor Analysis}
\author{%
Hien Duy Nguyen$^{1,2}$\thanks{Email: \href{mailto:h.nguyen5@latrobe.edu.au}{h.nguyen5@latrobe.edu.au}.}~ and Kei Hirose$^{1}$\thanks{Email: \href{mailto:hirose@imi.kyushu-u.ac.jp}{hirose@imi.kyushu-u.ac.jp}.}\\[1ex]
{\small $^1$Institute of Mathematics for Industry, Kyushu University, Fukuoka, Japan}\\
{\small $^2$Department of Mathematical and Physical Sciences, La Trobe University, Melbourne, Victoria, Australia}%
}

\makeatother

\begin{document}
\maketitle
\begin{abstract}
We study model selection by the Bayesian information criterion (BIC)
in fixed-dimensional exploratory factor analysis over a fixed finite
family of compact covariance classes. Our main result shows
that the BIC is strongly consistent for the pseudo-true factor
order under misspecification, provided that all globally optimal models
share a common pseudo-true covariance set, the population Gaussian
criterion has a local quadratic margin away from that set, and the
BIC complexity counts are order-separating at the pseudo-true order.
The candidate models may have an unknown mean vector, exact-zero restrictions
in the loading matrix, and either diagonal or spherical error covariance
structures, and the selection target is the smallest candidate factor
order that yields the best Gaussian approximation, in Kullback--Leibler
divergence, to the data-generating covariance structure. The proof works directly in covariance
space, so it does not require a regular loading parametrization and
accommodates the familiar singularities caused by rotations and redundant
factors. Under correct specification, the assumptions reduce to familiar
properties of the true covariance matrix. More generally, the same
argument applies to other information criteria whose penalties satisfy
the same gap conditions, including several BIC-type modifications.
\end{abstract}

\paragraph{Keywords.}

Bayesian information criterion; exploratory factor analysis; model
selection consistency; misspecified Gaussian likelihood; singular
models.

\section{Introduction}

\label{sec:intro}

Exploratory factor analysis remains one of the central tools of multivariate
analysis. In behavioural research, psychometrics, education, marketing,
finance, and many adjacent fields, it provides a parsimonious representation
of dependence among observed variables through a smaller number of
latent common factors. A useful general introduction is the monograph
of \citet{Basilevsky1994}. In practice, however, exploratory factor
analysis requires several model-selection decisions at once: one must
determine how many factors should be retained, what structural restrictions
should be placed on the loading matrix, and what form should be imposed
on the error covariance matrix.

In the classical factor model 
\[
\Sigma=\Lambda\Lambda^{\top}+\Psi,
\]
the analyst must determine the factor dimension, decide whether the
loading matrix $\Lambda$ is unrestricted or subject to exploratory
sparsity or confirmatory factor analysis (CFA) zero restrictions,
and choose the form of the error covariance matrix $\Psi$, typically
diagonal or spherical. A standard and practical way to make those
decisions is to compare candidate models by information criteria,
a tradition going back to \citet{Akaike1973} and \citet{Schwarz1978}
and extending early to factor analysis itself; see also \citet{Akaike1987},
who discussed the use of the Akaike information criterion (AIC) in
factor models. %The factor-analysis literature using such criteria is well reviewed in \citet{HiroseEtAl2011} and \citet{ChoiJeong2019}. In parallel, regularization has been used to encourage sparsity, simplified loading structure, and numerical stability through penalized likelihood; see, for example, \citet{HiroseImada2018} and \citet{HiroseTerada2023}.
A Bayesian approach to model assessment in factor analysis is also
discussed by \citet{LopesWest2004}. Specific work on selecting the
number of factors by information criteria includes \citet{HiroseEtAl2011}.
Related work for approximate factor models includes \citet{ChoiJeong2019}.

Despite the routine use of the Bayesian information criterion (BIC)
in practical factor-order selection (see e.g., \citealp{LopesWest2004,HiroseEtAl2011,SongBelin2008,ChenEtAl2010,PearsonEtAl2013,PreacherEtAl2013}),
its theoretical justification in such applications is delicate. The
classical asymptotic arguments behind BIC in regular or misspecified
parametric families, including the model-selection analyses of \citet{Nishii1988}
and \citet{Vuong1989} and the broader misspecified or loss-based treatments
of \citet{SinWhite1996} and \citet{Nguyen2024}, rely on local regularity
around pseudo-true minimizers: smooth local parametrizations, well-behaved
quadratic expansions, and nondegenerate curvature. Related minimum
empirical risk asymptotics on general parameter spaces are developed
by \citet{WesterhoutEtAl2024}. Factor models do not generally satisfy
those regularity requirements. They are often singular latent-variable
models in precisely the sense emphasized by \citet{Watanabe2009}
and \citet{DrtonPlummer2017}. When one compares models of different
factor dimension, the lower-order model typically lies on a singular
manifold of the higher-order model, where redundant loadings vanish,
rotations become unidentified, and information matrices may be singular.
\citet{ChenMoustakiZhang2020} make this issue especially clear for
likelihood-ratio tests. The same singular geometry explains why the
ordinary regular-model justification of BIC cannot simply be applied
without modification to exploratory factor analysis. Throughout, we
use ``BIC'' in the broad model-selection sense of the criterion obtained
by penalizing the profiled log-likelihood by $\frac{1}{2}d_{k}\log n$
for a chosen complexity assignment $d_{k}$. In singular settings this
broad usage need not coincide with the narrower Laplace-approximation-based
Bayesian marginal-likelihood interpretation; see \citet{Watanabe2013} and \citet{DrtonPlummer2017}.

Recent works show that the consistency of information criteria depends
strongly on the asymptotic regime and on the statistical object being
selected. \citet{BaiNg2002} study static approximate factor models
when both the cross-sectional and time dimensions diverge. \citet{AmengualWatson2007}
analyze the number of dynamic factors in large panels, and \citet{ChoiJeong2019}
derive and compare AIC-, CAIC-, BIC-, and Hannan--Quinn (HQ)-type
criteria for approximate factor models in the same large-panel tradition.
Related refinements or alternatives include the spectral-density criterion
of \citet{HallinLiska2007} for the general dynamic factor model,
the tuning refinement of \citet{AlessiBarigozziCapasso2010} for the
criterion of \citet{BaiNg2002}, the edge-distribution method of \citet{Onatski2010},
and the eigenvalue-ratio and growth-ratio selectors of \citet{AhnHorenstein2013}.
In structural equation modeling more broadly, \citet{Huang2017} analyzes
the asymptotic behavior of AIC, BIC, and the root mean square error
of approximation (RMSEA) under weak distributional assumptions. In
high-dimensional rank-selection problems, \citet{MorimotoHungHuang2026}
obtain information-criterion consistency results under random-matrix
gap conditions. Thus these contributions address different asymptotic
regimes or different selection targets from the fixed-$p$ Gaussian
factor-analysis covariance classes of direct interest here, so they
do not directly yield the theorem proved below for the compact exploratory
setting.

Another related branch of the literature approaches the problem from
singular Bayesian asymptotics. \citet{DrtonPlummer2017} develop the
singular BIC (sBIC) for singular model-selection problems, \citet{Watanabe2013}
develops the widely applicable BIC (WBIC) within singular learning
theory, and \citet{DrtonEtAl2025} sharpen the singular-learning analysis
specifically for factor analysis. This literature clarifies that,
in singular models, the Bayesian marginal-likelihood interpretation
of BIC is subtler than in regular settings. But it answers a different
question from the one pursued here. Our consistency target
is the smallest factor order among the Gaussian Kullback--Leibler-optimal
covariance classes; under the BIC corollary, the selected model is
then of minimal BIC complexity within that optimal order. Singular
learning theory instead studies asymptotic marginal-likelihood penalties
that characterize complexity via intrinsic learning coefficients,
which are model-dependent geometric quantities and are not themselves
the frequentist order target considered here.

A complementary generic route is provided by recent information-criterion
results beyond likelihood-based factor models. \citet{Nguyen2024} and \citet{WesterhoutEtAl2024}
develop the PanIC and Sin--White information criterion (SWIC) framework, which gives general sufficient conditions
for consistent information criteria in loss-based model selection. In the minimum-empirical-risk normalization of
\citet{WesterhoutEtAl2024}, BIC contributes a penalty of
order $(\log n)/n$, whereas the SWIC penalty is of order $(\log n)/\sqrt{n}$.
Thus the generic sufficient conditions there require penalties that
are asymptotically larger than BIC; on the unnormalized log-likelihood
scale used in the present work, the corresponding SWIC penalty is
of order $\sqrt{n}\log n$. Consequently, these generic empirical-risk
results do not by themselves yield the BIC theorem proved here.

The main contribution of this work is an empirical process-based theorem
for structured exploratory factor analysis over a fixed finite candidate
family of compact covariance classes. When the mean vector is unknown,
it profiles out exactly, because for every covariance matrix the Gaussian
likelihood is maximized at the sample mean. The problem therefore
reduces to choosing the covariance structure. Under unrestricted misspecification,
BIC may fail to be consistent. We therefore isolate a natural
structural assumption whereby every globally optimal model, and in
particular every globally optimal higher-order model (an exact overfit
in our terminology), shares a common pseudo-true covariance set already
attained at the minimal optimal order. Together with a local quadratic
margin around that common set, and with order-separating BIC complexity
counts at the pseudo-true order, this structure is enough to recover
strong consistency of BIC over the intended family of restrictions.
The proof follows the underfit/overfit template of \citet{Keribin2000}
for mixture models, but it is carried out in covariance space and
uses compact empirical-process localization arguments in the spirit
of \citet{vanDeGeer2000}. Related empirical-process analyses of nonregular
order selection are given by \citet{vanHandel2011} and \citet{GassiatVanHandel2013}.

Three further features of the theorem are worth highlighting. First,
the target under misspecification is defined through Gaussian Kullback--Leibler
approximation in the sense of pseudo-true covariance sets, following
the misspecified-likelihood framework of \citet{White1982,WhiteBook1994}.
Under correct specification, this target reduces exactly to the classical
desideratum: the pseudo-true order is then the smallest order whose
covariance class contains the true covariance matrix. Thus the present
framework extends, rather than replaces, the usual correctly specified
interpretation. Second, the proof is carried out in covariance space
rather than loading space. That is crucial because the loading parametrization
is where rotational nonidentifiability, self-intersections, exact
zeros, and redundant factors generate singular geometry. By working
instead with the covariance representation, we allow those singularities
rather than rule them out by assumption. Third, the main theorem is
formulated for a broader class of penalties than those defined by
the usual BIC multiplier $(1/2)\log n$. BIC is recovered as the leading
corollary once the corresponding complexity counts separate the pseudo-true
order from higher-order optimal models, but the same argument applies
to any penalty system whose relevant gaps between higher-order exact
overfits and the minimal optimal order diverge faster than $\log\log n$
while remaining $o(n)$. Thus the paper characterizes a whole family
of BIC-type and heavier sublinear information criteria, including
many penalties considered by \citet{Huang2017} and \citet{ChoiJeong2019}.

The remainder of the paper is organized as follows. Section~\ref{sec:notation}
introduces the candidate covariance classes, the profiled Gaussian
quasi-likelihood with unknown mean, the pseudo-true order, and the
misspecification assumptions that isolate the exact-overfit geometry
needed for the theorem. Section~\ref{sec:main} states the main consistency
theorem, gives its BIC corollaries for dense and sparse structured
factor classes, and records conventional complexity assignments together
with more general sublinear penalty variants. Section~\ref{sec:discussion}
compares the theorem with the surrounding literature and explains
more precisely what is and is not covered by earlier positive results,
by regularization methods, and by singular Bayesian criteria such
as sBIC and WBIC. The appendix includes auxiliary technical results
and proofs of the stated theorems and propositions.

\section{Notation and technical preliminaries}

\label{sec:notation}

\subsection{Mathematical setup and notation}

Throughout, $(\Omega,\mathcal{A},\mathbb{P})$ is the underlying probability
space. All random variables are defined on this space, and all probabilities
and almost-sure statements are taken with respect to $\mathbb{P}$.
The data are an iid sequence $X_{1},X_{2},\dots:\Omega\to\mathbb{R}^{p}$
with common law $P_{0}$, mean $\mu_{0}\in\mathbb{R}^{p}$, and covariance
$\Sigma_{0}\in\mathbb{S}_{++}^{p}$. For an integer $r\ge1$, write
$[r]=\{1,\dots,r\}$, and adopt the convention $[0]=\varnothing$.

Let $\mathbb{S}^{p}$ denote the space of real symmetric $p\times p$
matrices, let $\mathbb{S}_{+}^{p}$ denote the cone of positive semidefinite
matrices, and let $\mathbb{S}_{++}^{p}$ denote the cone of positive
definite matrices. For $A,B\in\mathbb{S}^{p}$, we write $A\preceq B$
if $B-A\in\mathbb{S}_{+}^{p}$ and $A\prec B$ if $B-A\in\mathbb{S}_{++}^{p}$;
this is the usual Loewner ordering. The identity matrix is denoted
by $I_{p}$. For vectors $x\in\mathbb{R}^{p}$, $\|x\|$ denotes
the Euclidean norm. We use $\|\cdot\|_{F}$ and $\|\cdot\|_{\mathrm{op}}$
for the Frobenius and operator norms, respectively, and distances
between matrices are taken with respect to $\|\cdot\|_{F}$ unless
explicitly stated otherwise. For $\mu\in\mathbb{R}^{p}$ and $\Sigma\in\mathbb{S}_{++}^{p}$,
$\mathcal{N}(\mu,\Sigma)$ denotes the $p$-variate normal law with
mean $\mu$ and covariance $\Sigma$. Thus, for a set $\mathcal{S}\subset\mathbb{S}^{p}$,
\[
\dist(\Sigma,\mathcal{S})=\inf_{\Gamma\in\mathcal{S}}\|\Sigma-\Gamma\|_{F}.
\]
Because $p$ is fixed throughout, all norms on $\mathbb{S}^{p}$ are
equivalent; the Frobenius norm is used only for convenience.

\subsection{Structured factor-analysis covariance classes}

Fix an observed dimension $p\ge1$ and a maximal candidate order $q_{\max}\ge0$.
Let $m\ge1$ be the total number of candidate covariance models. For
each model index $k\in[m]$, fix a factor order $q_{k}\in\{0,1,\dots,q_{\max}\}$,
a support pattern $E_{k}\subseteq[p]\times[q_{k}]$, and an error-covariance
type $\tau_{k}\in\{\mathrm{diag},\mathrm{sph}\}$. Thus model $k$
stands for a fully specified structured factor-analysis covariance
class.

Fix constants 
\[
0<\underline{\psi}<\overline{\psi}<\infty,\qquad M<\infty.
\]
For each $k$, define the loading class 
\[
\mathcal{L}_{k}=\left\{\Lambda\in\mathbb{R}^{p\times q_{k}}:\lambda_{jh}=0\text{ whenever }(j,h)\notin E_{k},\ \|\Lambda\|_{F}\le M\right\},
\]
and define the corresponding error-covariance classes 
\[
\mathcal{P}_{\mathrm{diag}}=\{\diag(\psi_{1},\dots,\psi_{p}):\underline{\psi}\le\psi_{j}\le\overline{\psi}\},\qquad\mathcal{P}_{\mathrm{sph}}=\{\psi I_{p}:\underline{\psi}\le\psi\le\overline{\psi}\}.
\]
The covariance class attached to model $k$ is 
\[
\mathcal{F}_{k}=\{\Sigma=\Lambda\Lambda^{\top}+\Psi:\Lambda\in\mathcal{L}_{k},\ \Psi\in\mathcal{P}_{\tau_{k}}\}\subset\mathbb{S}_{++}^{p}.
\]
These classes include compactly bounded dense exploratory models, sparse
exploratory models, and confirmatory zero patterns, together with
either diagonal or spherical error covariance matrices.

\begin{proposition}[Compactness and uniform eigenvalue bounds]\label{prop:compact}
For every $k\in[m]$, the class $\mathcal{F}_{k}$ is compact. The
finite union 
\[
\mathcal{K}=\bigcup_{k=1}^{m}\mathcal{F}_{k}
\]
is compact as well, and every $\Sigma\in\mathcal{K}$ satisfies 
\[
\underline{\psi}I_{p}\preceq\Sigma\preceq(M^{2}+\overline{\psi})I_{p}.
\]
Consequently the inverse map is uniformly bounded and Lipschitz on
$\mathcal{K}$. \end{proposition}

\subsection{Gaussian quasi-likelihood}

Let $X_{1},\dots,X_{n}$ be iid observations in $\mathbb{R}^{p}$.
For $\mu\in\mathbb{R}^{p}$ and $\Sigma\in\mathbb{S}_{++}^{p}$, define
the Gaussian log-likelihood (up to the irrelevant constant $-np\log(2\pi)/2$)
by 
\[
\ell_{n}(\mu,\Sigma)=-\frac{n}{2}\log\det\Sigma-\frac{1}{2}\sum_{t=1}^{n}(X_{t}-\mu)^{\top}\Sigma^{-1}(X_{t}-\mu).
\]
Write 
\[
\bar{X}_{n}=\frac{1}{n}\sum_{t=1}^{n}X_{t},\qquad S_{n}=\frac{1}{n}\sum_{t=1}^{n}(X_{t}-\bar{X}_{n})(X_{t}-\bar{X}_{n})^{\top}.
\]

\begin{proposition}[Profiling the mean]\label{prop:profile} For
every fixed $\Sigma\in\mathbb{S}_{++}^{p}$, the map $\mu\mapsto\ell_{n}(\mu,\Sigma)$
is uniquely maximized at $\bar{X}_{n}$. Hence the profiled Gaussian
criterion is 
\[
\widetilde{\ell}_{n}(\Sigma)=\sup_{\mu\in\mathbb{R}^{p}}\ell_{n}(\mu,\Sigma)=-\frac{n}{2}\left\{\log\det\Sigma+\tr(S_{n}\Sigma^{-1})\right\}.
\]
For each $k\in[m]$, define 
\[
T_{k,n}=\sup_{\Sigma\in\mathcal{F}_{k}}\widetilde{\ell}_{n}(\Sigma).
\]
Because $\mathcal{F}_{k}$ is compact and $\widetilde{\ell}_{n}$
is continuous, the supremum is attained. \end{proposition}

A useful consequence is that if the BIC is written for the full Gaussian
parameter pair $(\mu,\Sigma)$ with unknown mean, every candidate
model receives the same additional mean penalty $p\log n/2$. Therefore
order comparisons are unchanged whether one works with the full-parameter
unknown-mean BIC or the profiled unknown-mean BIC. This observation
does not identify the profiled criterion with the known-mean criterion, whose likelihood differs
from $\widetilde{\ell}_{n}$ by a $\Sigma$-dependent term.

\subsection{Population quantities}

Assume from now on that the data-generating law $P_{0}$ has mean
$\mu_{0}$ and covariance $\Sigma_{0}\in\mathbb{S}_{++}^{p}$. The
population Gaussian cross-entropy corresponding to $\widetilde{\ell}_{n}$
is 
\[
Q(\Sigma)=-\frac{1}{2}\left\{\log\det\Sigma+\tr(\Sigma_{0}\Sigma^{-1})\right\}.
\]
Equivalently,
\[
Q(\Sigma)=-\KL\!\left(\mathcal{N}(\mu_{0},\Sigma_{0})\,\|\,\mathcal{N}(\mu_{0},\Sigma)\right)-\frac{1}{2}\left(\log\det\Sigma_{0}+p\right),
\]
so maximizing $Q$ is the same as minimizing the Gaussian Kullback--Leibler
divergence from the covariance structure $(\mu_{0},\Sigma_{0})$ to
$(\mu_{0},\Sigma)$. For the full unknown-mean criterion, 
\[
\Gamma(\mu,\Sigma)=-\frac{1}{2}\left\{\log\det\Sigma+\tr(\Sigma_{0}\Sigma^{-1})+(\mu-\mu_{0})^{\top}\Sigma^{-1}(\mu-\mu_{0})\right\},
\]
so the population-optimal mean is $\mu_{0}$ for every $\Sigma$.

\begin{definition}[Pseudo-true covariance sets and order]\label{def:pseudo}
For each $k\in[m]$, define 
\[
V_{k}=\sup_{\Sigma\in\mathcal{F}_{k}}Q(\Sigma),\qquad\mathcal{G}_{k}=\argmax_{\Sigma\in\mathcal{F}_{k}}Q(\Sigma).
\]
Let 
\[
V_{*}=\max_{1\le k\le m}V_{k},\qquad K_{*}=\{k\in[m]:V_{k}=V_{*}\}.
\]
The pseudo-true order is 
\[
q_{*}=\min\{q_{k}:k\in K_{*}\},
\]
and the set of globally optimal models of minimal order is 
\[
K_{0}=\{k\in K_{*}:q_{k}=q_{*}\}.
\]
\end{definition}

We shall call a model $k\in K_{*}$ with $q_{k}>q_{*}$ an exact
overfit (or exact overfitted model). Thus an exact overfit is a globally
optimal model of strictly larger order than the pseudo-true order.
Under Assumption~\ref{it:M2}, every exact overfit shares the common
pseudo-true covariance set $\mathcal{G}_{*}$.

By continuity of $Q$ and compactness of each $\mathcal{F}_{k}$,
every pseudo-true set $\mathcal{G}_{k}$ is nonempty and compact.
Thus the selection target under misspecification is not a true parametric
order, but the smallest factor order attaining the best Gaussian Kullback--Leibler
approximation among the candidate covariance classes.

\begin{remark}[Correct specification as a special case]\label{rem:correctspecial}
If the true covariance matrix $\Sigma_{0}$ belongs to at least one
candidate class $\mathcal{F}_{k}$ (in particular, under Gaussian
correct specification), then $Q$ is uniquely maximized at $\Sigma_{0}$
in ambient covariance space. In that case every exact overfit model
has pseudo-true set $\mathcal{G}_{k}=\{\Sigma_{0}\}$, so $q_{*}$
coincides with the minimal correctly specified order. \end{remark}

\subsection{Technical assumptions}

Our theory requires the following regularity assumption.

\begin{assumption}\label{ass:miss} \leavevmode 
\begin{enumerate}[label=\textnormal(M\arabic*)]
\item \label{it:M1} The observations are iid with mean $\mu_{0}$, covariance
$\Sigma_{0}\in\mathbb{S}_{++}^{p}$, and finite fourth moment $\mathbb{E}\|X_{1}\|^{4}<\infty$. 
\item \label{it:M2} There exists a nonempty compact set $\mathcal{G}_{*}\subset\mathcal{K}$
such that 
\[
\mathcal{G}_{k}=\mathcal{G}_{*}\qquad\text{for every }k\in K_{*}.
\]
\item \label{it:M3} For each $k\in K_{*}$ there exist constants $c_{k}>0$
and $\eta_{k}>0$ such that 
\[
Q(\Sigma)\le V_{*}-c_{k}\,\dist(\Sigma,\mathcal{G}_{*})^{2}\qquad\text{for all }\Sigma\in\mathcal{F}_{k}\text{ with }\dist(\Sigma,\mathcal{G}_{*})<\eta_{k}.
\]
\end{enumerate}
\end{assumption}

Assumption~\ref{it:M2} requires all globally optimal models, and
in particular all exact overfits, to share a common pseudo-true covariance
set. It is the misspecified analogue of the correctly specified fact
that all exact overfits contain the same true covariance matrix. Assumption~\ref{it:M3}
is the local quadratic margin that converts this common-projection
property into an $O(\log\log n)$ bound for the empirical gain of
an exact overfit.

\begin{remark}[A sufficient condition for \ref{it:M3}]\label{rem:M3auto}
Assumption~\ref{it:M3} holds in the correctly specified setting
and, more generally, whenever the pseudo-true covariance set reduces
to a single covariance at which $Q$ has a strictly negative quadratic
expansion. In particular, if $\Sigma_{0}\in\mathcal{F}_{k}$ for at
least one candidate model, then $Q$ is uniquely maximized at $\Sigma_{0}$
in ambient covariance space, so every exact overfitted model has pseudo-true
set $\{\Sigma_{0}\}$. Moreover, the second derivative 
\[
D^{2}Q(\Sigma_{0})[H,H]=-\frac{1}{2}\tr\!\left(\Sigma_{0}^{-1}H\Sigma_{0}^{-1}H\right),\qquad H\in\mathbb{S}^{p},
\]
is strictly negative for every nonzero symmetric direction $H$, because
\[
\tr\!\left(\Sigma_{0}^{-1}H\Sigma_{0}^{-1}H\right)=\left\|\Sigma_{0}^{-1/2}H\Sigma_{0}^{-1/2}\right\|_{F}^{2}>0
\]
whenever $H\neq0$. By continuity of $D^{2}Q$, there exist constants
$c>0$ and $\eta>0$ such that
\[
Q(\Sigma)\le Q(\Sigma_{0})-c\,\|\Sigma-\Sigma_{0}\|_{F}^{2}\qquad\text{whenever }\|\Sigma-\Sigma_{0}\|_{F}<\eta.
\]
Since $\dist(\Sigma,\{\Sigma_{0}\})=\|\Sigma-\Sigma_{0}\|_{F}$,
this is exactly the quadratic margin in \ref{it:M3}. Thus \ref{it:M3}
is not arbitrary: it characterizes the natural misspecified analogue
of a property that is immediate under correct specification. \end{remark}

\begin{remark}[On condition \ref{it:M2}]\label{rem:whynecessary}
Heuristically, if two equally good higher-order models are optimized
at different isolated pseudo-true covariances $\Sigma_{i}^{*}\neq\Sigma_{j}^{*}$,
then 
\[
\widetilde{\ell}_{n}(\Sigma_{i}^{*})-\widetilde{\ell}_{n}(\Sigma_{j}^{*})=-\frac{n}{2}\tr\left((S_{n}-\Sigma_{0})(\Sigma_{i}^{*-1}-\Sigma_{j}^{*-1})\right),
\]
which, whenever the corresponding centered linear contrast is nondegenerate,
is generically of order $\sqrt{n}$ in probability and of order $\sqrt{n\log\log n}$
almost surely. A BIC penalty of order $\log n$ cannot dominate
such fluctuations. Thus one should not expect a universal misspecified
BIC theorem without a condition such as Assumption~\ref{it:M2}. \end{remark}

\begin{remark}[Pathologies when \ref{it:M2} or \ref{it:M3} fail]\label{rem:pathologies} 
It is useful to distinguish the two mechanisms that Assumption~\ref{ass:miss}
rules out. The first is the failure of a common pseudo-true covariance
set. In a one-dimensional covariance problem, take $p=1$ and
$\Sigma_{0}=1$, so
\[
Q(\sigma)=-\frac{1}{2}\left\{\log\sigma+\frac{1}{\sigma}\right\},\qquad\sigma>0.
\]
Here
\[
Q'(\sigma)=-\frac{1}{2}\left(\frac{1}{\sigma}-\frac{1}{\sigma^{2}}\right)=-\frac{\sigma-1}{2\sigma^{2}},
\]
so $Q$ is strictly increasing on $(0,1)$ and strictly decreasing on
$(1,\infty)$. Fix any $\sigma_{-}\in(0,1)$. Since $Q(\sigma)\to-\infty$
as $\sigma\downarrow0$ or $\sigma\uparrow\infty$, there exists a unique
$\sigma_{+}>1$ such that $Q(\sigma_{+})=Q(\sigma_{-})$. If one takes the compact covariance classes $\mathcal{F}_{1}=\{\sigma_{-}\}$ and
$\mathcal{F}_{2}=\{\sigma_{+}\}$, then $K_{*}=\{1,2\}$ but
$\mathcal{G}_{1}\neq\mathcal{G}_{2}$. Since each class is a singleton,
\[
T_{1,n}=\widetilde{\ell}_{n}(\sigma_{-}),\qquad T_{2,n}=\widetilde{\ell}_{n}(\sigma_{+}),
\]
and
\[
\widetilde{\ell}_{n}(\sigma_{-})-\widetilde{\ell}_{n}(\sigma_{+})=-\frac{n}{2}\left\{\log\sigma_{-}-\log\sigma_{+}+S_{n}\left(\sigma_{-}^{-1}-\sigma_{+}^{-1}\right)\right\}.
\]
Because $Q(\sigma_{+})=Q(\sigma_{-})$, we also have
\[
\log\sigma_{-}-\log\sigma_{+}=\sigma_{+}^{-1}-\sigma_{-}^{-1},
\]
so the likelihood contrast simplifies to
\[
\widetilde{\ell}_{n}(\sigma_{-})-\widetilde{\ell}_{n}(\sigma_{+})=-\frac{n}{2}(S_{n}-1)\left(\sigma_{-}^{-1}-\sigma_{+}^{-1}\right).
\]
The coefficient is nonzero because $\sigma_{-}\neq\sigma_{+}$. Hence the
likelihood difference has exactly the fluctuation scale of the centered
sample covariance $S_{n}-1$; in generic nondegenerate cases this is of
order $\sqrt{n}$ in probability and of order $\sqrt{n\log\log n}$
almost surely, which is too large for a BIC penalty of order $\log n$
to dominate.

The second pathology is the failure of the quadratic margin even when
the pseudo-true covariance set is common. Assume here that $p\ge2$, so
that covariance space has dimension at least two. Fix any covariance
$\Sigma^{\star}\in\mathbb{S}_{++}^{p}$ with $\Sigma^{\star}\neq\Sigma_{0}$.
Since
\[
DQ(\Sigma)[H]=-\frac{1}{2}\tr\left(\left(\Sigma^{-1}-\Sigma^{-1}\Sigma_{0}\Sigma^{-1}\right)H\right),\qquad H\in\mathbb{S}^{p},
\]
the derivative vanishes only at $\Sigma=\Sigma_{0}$, so $DQ(\Sigma^{\star})\neq0$.
By the implicit function theorem, the level set $\{\Sigma:Q(\Sigma)=Q(\Sigma^{\star})\}$
therefore contains a smooth local curve $\zeta(t)$ through $\Sigma^{\star}$
with $\zeta(0)=\Sigma^{\star}$ and $\|\zeta(t)-\Sigma^{\star}\|_{F}\asymp|t|$.
Because $DQ(\Sigma^{\star})$ is a nonzero linear functional, one can
choose $H\in\mathbb{S}^{p}$ with $DQ(\Sigma^{\star})[H]<0$ and define
\[
\gamma(t)=\zeta(t)+t^{4}H.
\]
For $|t|$ small enough, $\gamma(t)$ remains positive definite, and a
first-order Taylor expansion around $\zeta(t)$ gives
\[
Q(\gamma(t))=Q(\zeta(t))+t^{4}DQ(\zeta(t))[H]+o(t^{4})=Q(\Sigma^{\star})+t^{4}DQ(\Sigma^{\star})[H]+o(t^{4})
\]
because $Q(\zeta(t))=Q(\Sigma^{\star})$ and $DQ(\zeta(t))[H]=DQ(\Sigma^{\star})[H]+o(1)$.
Thus
\[
Q(\gamma(t))=Q(\Sigma^{\star})-c t^{4}+o(t^{4})
\]
for some $c>0$, while $\|\gamma(t)-\Sigma^{\star}\|_{F}\asymp|t|$. Hence, for $\delta>0$ small, the compact covariance class
\[
\mathcal{F}=\{\gamma(t):|t|\le\delta\}
\]
has pseudo-true set $\{\Sigma^{\star}\}$ but satisfies
\[
V_{*}-Q(\gamma(t))\asymp t^{4}=o\left(\dist(\gamma(t),\{\Sigma^{\star}\})^{2}\right),
\]
so Assumption~\ref{it:M3} fails on $\mathcal{F}$. If one includes alongside $\mathcal{F}$ the singleton class $\{\Sigma^{\star}\}$, then the globally optimal models share the common pseudo-true set $\{\Sigma^{\star}\}$ while \ref{it:M3} still fails for the curved class. These examples are purely illustrative,
but they show that \ref{it:M2} and \ref{it:M3} exclude different kinds of
local pathology.
\end{remark}

\section{Main results}

\label{sec:main}

Let $a_{k,n}\in\mathbb{R}$ be the penalty attached to model $k$.
Define the penalized score 
\[
W_{k,n}=T_{k,n}-a_{k,n},\qquad k\in[m].
\]
The selected model and factor order are 
\[
\widehat{K}_{n}=\min\argmax_{1\le k\le m}W_{k,n},\qquad\widehat{q}_{n}=q_{\widehat{K}_{n}}.
\]
Also define the optimal empirical benchmark 
\[
U_{n}=\sup_{\Sigma\in\mathcal{G}_{*}}\widetilde{\ell}_{n}(\Sigma),\qquad a_{n}^{*}=\min\{a_{k,n}:k\in K_{0}\}.
\]

The next theorem is the central result of our work.

\begin{theorem}[Consistency under misspecification]\label{thm:main}
Assume Assumption~\ref{ass:miss}. Let $(a_{k,n})_{k\in[m],n\ge1}$
satisfy 
\begin{enumerate}[label=\textnormal(P\arabic*)]
\item \label{it:P1} $a_{k,n}=o(n)$ for every $k\in[m]$; 
\item \label{it:P2} $a_{k,n}-a_{n}^{*}\to\infty$ for every $k\in K_{*}$
with $q_{k}>q_{*}$; 
\item \label{it:P3} $\log\log n\,/\,(a_{k,n}-a_{n}^{*})\to0$ for every
$k\in K_{*}$ with $q_{k}>q_{*}$. 
\end{enumerate}
Then 
\[
\widehat{q}_{n}\to q_{*}\qquad\text{almost surely.}
\]
\end{theorem}

The proof follows an empirical process-based underfit/overfit template,
conceptually close to \citet{Keribin2000} and implemented through
compact covariance-space localization in the spirit of \citet{vanDeGeer2000}.
First, suboptimal models are excluded because their population criterion
lies strictly below $V_{*}$, and the resulting empirical likelihood
loss is linear in $n$. Second, exact overfits are controlled by the
common-projection property and the local quadratic margin, which imply
that their empirical likelihood gain is at most $O(\log\log n)$ almost
surely. Any penalty system whose relevant gaps $a_{k,n}-a_{n}^{*}$
for $k\in K_{*}$ with $q_{k}>q_{*}$ dominate $\log\log n$, while
all penalties remain $o(n)$, therefore forces eventual selection of
a minimal-order optimal model.

\subsection{Consistency of the BIC}

To apply Theorem~\ref{thm:main} to BIC, let $d_{k}>0$
be a fixed covariance-complexity count for model $k$. The profiled
BIC score is equivalent to maximizing 
\[
T_{k,n}-\frac{1}{2}d_{k}\log n,
\]
because the common mean contribution $p\log n/2$ is the same for
every candidate model. Define 
\[
d_{*}=\min\{d_{k}:k\in K_{0}\},\qquad K_{**}=\{k\in K_{0}:d_{k}=d_{*}\}.
\]

\begin{corollary}[BIC consistency]\label{cor:bic} Under Assumption~\ref{ass:miss},
suppose the BIC complexity system is order-separating at the pseudo-true
order in the sense that 
\[
d_{k}>d_{*}\qquad\text{for every }k\in K_{*}\text{ with }q_{k}>q_{*}.
\]
Then the BIC-selected order is strongly consistent: 
\[
\widehat{q}_{n}^{\mathrm{BIC}}\to q_{*}\qquad\text{almost surely.}
\]
\end{corollary}

\begin{remark}[Model-index interpretation]\label{rem:Kstarstar}
The order conclusion of Corollary~\ref{cor:bic} is the main selection
statement. A sharper model-index conclusion is that 
\[
\mathbb{P}\left(\widehat{K}_{n}^{\mathrm{BIC}}\in K_{**}\text{ for all sufficiently large }n\right)=1.
\]
Thus the BIC eventually selects from the class of globally optimal
models that have both the minimal optimal order and the smallest complexity
count within that order. If $K_{**}$ is a singleton, then the model
index itself is strongly consistent. If $|K_{**}|>1$, then the candidate
family contains several asymptotically indistinguishable models with
the same optimal order and the same BIC complexity. In that case
the theory guarantees only eventual selection within $K_{**}$; unique
model-index consistency would require $K_{**}$ to be a singleton
or additional structure that separates those models. \end{remark}

\subsection{Conventional complexity assignments for common structured factor classes}

The theorem is abstract in the sense that any fixed complexity assignment
$d_{k}$ may be used. In practice, one often chooses $d_{k}$ from a
conventional parameter-counting scheme tied to an identification of the
loading model. The theorem itself only needs the resulting numerical
weights. In the broad sense explained in the Introduction, we continue
to refer to the resulting criterion $T_{k,n}-\frac{1}{2}d_{k}\log n$
as BIC.

\paragraph{Dense exploratory classes.}

For the dense support pattern $E_{\mathrm{full},q}=[p]\times[q]$,
a standard lower-triangular loading gauge with positive diagonal entries
in the leading $q\times q$ block yields the conventional loading-coordinate
count
\[
pq-\frac{q(q-1)}{2}.
\]
Appending either $p$ diagonal uniqueness parameters or one spherical
uniqueness parameter gives the numerical weights
\[
d_{q,\mathrm{full},\mathrm{diag}}=pq-\frac{q(q-1)}{2}+p,\qquad d_{q,\mathrm{full},\mathrm{sph}}=pq-\frac{q(q-1)}{2}+1.
\]
The theorem below uses these numbers as fixed complexity weights in
the corresponding BIC formula. The order gaps satisfy 
\[
d_{q+1,\mathrm{full},\tau}-d_{q,\mathrm{full},\tau}=p-q.
\]
Thus whenever $0\le q\le q_{\max}\le p-1$, these numerical weights are strictly
increasing in $q$, so the relevant order-separation condition is automatic
within a fixed error-covariance type.

\begin{corollary}[Dense exploratory families]\label{cor:dense}
Consider the dense exploratory family of orders $q=0,1,\dots,q_{\max}\le p-1$
with either diagonal error covariance matrices for all models or spherical
error covariance matrices for all models. Equip the models with the
corresponding conventional dense weights displayed above. If Assumption~\ref{ass:miss}
holds for this family, then the corresponding BIC selector satisfies
\[
\widehat{q}_{n}\to q_{*}\qquad\text{almost surely.}
\]
\end{corollary}
\begin{proof}
Fix the common error-covariance type $\tau\in\{\mathrm{diag},\mathrm{sph}\}$.
For dense exploratory models of order $q$, the numerical weight is
\[
d_{q,\mathrm{full},\tau}=pq-\frac{q(q-1)}{2}+c_{\tau},
\]
where $c_{\mathrm{diag}}=p$ and $c_{\mathrm{sph}}=1$. Hence, for
every $0\le q\le q_{\max}-1$, 
\[
d_{q+1,\mathrm{full},\tau}-d_{q,\mathrm{full},\tau}=p-q>0
\]
because $q_{\max}\le p-1$. Thus $q\mapsto d_{q,\mathrm{full},\tau}$ is
strictly increasing on $\{0,1,\dots,q_{\max}\}$. In particular, if $k\in K_{*}$
has $q_{k}>q_{*}$, then $d_{k}>d_{j}$ for every $j\in K_{0}$, so
$d_{k}>d_{*}$. Therefore the order-separation hypothesis needed for the
selector with penalty $\tfrac{1}{2}d_{k}\log n$ is automatic, and Theorem~\ref{thm:main}
yields the stated convergence. 
\end{proof}

\paragraph{Sparse exploratory and confirmatory classes.}

If a support pattern $E\subseteq[p]\times[q]$ together with a chosen
identification scheme yields a local loading chart of dimension $r_{q,E}$,
that is, near a regular admissible loading matrix the constrained
loading manifold can be parametrized by $r_{q,E}$ free local coordinates
after imposing a gauge that removes rotational indeterminacy, then
one may use 
\[
d(q,E,\mathrm{diag})=r_{q,E}+p,\qquad d(q,E,\mathrm{sph})=r_{q,E}+1.
\]
For example, in a two-factor model whose support allows a $2\times2$
anchor block, one may impose on that block the local lower-triangular
gauge
\[
\left(\begin{array}{cc}
\lambda_{11} & 0\\
\lambda_{21} & \lambda_{22}
\end{array}\right),\qquad \lambda_{11}>0,\ \lambda_{22}>0.
\]
Near a regular loading with $\lambda_{11}\lambda_{22}\neq0$, this removes
the one-dimensional rotational indeterminacy, and the remaining supported
entries provide local coordinates.
If one prefers not to commit to a particular identification gauge,
one may instead use the raw support count 
\[
d(q,E,\mathrm{diag})=|E|+p,\qquad d(q,E,\mathrm{sph})=|E|+1.
\]
The latter is conservative when rotational nonidentifiability remains.
Nevertheless it is still covered by Theorem~\ref{thm:main}, because
the argument only requires a fixed complexity system whose penalty
gaps separate the relevant optimal classes.

\begin{corollary}[Sparse and structured loadings]\label{cor:sparse}
Consider any finite family of sparse exploratory or confirmatory factor-analysis
covariance classes of the form described above. Let $d_{k}$ be any
fixed complexity assignment, for example an identified-chart count
when such a chart is available, or a raw-support count. If this assignment is order-separating at
the pseudo-true order, then the corresponding BIC selector satisfies 
\[
\widehat{q}_{n}\to q_{*}\qquad\text{almost surely.}
\]
Moreover the selected model eventually belongs almost surely to the
minimal-complexity optimal class $K_{**}$. \end{corollary}
\begin{proof}
Apply Theorem~\ref{thm:main} with the penalty system $a_{k,n}=\tfrac{1}{2}d_{k}\log n$.
The assumed order-separation property gives conditions \ref{it:P2}--\ref{it:P3},
and \ref{it:P1} is immediate because $\log n=o(n)$. The final statement
about eventual membership in $K_{**}$ follows from the same model-by-model
comparison used in the proof of Corollary~\ref{cor:bic}. 
\end{proof}

\subsection{Other admissible penalties}

\label{subsec:penalties}

Theorem~\ref{thm:main} is not tied to the usual BIC multiplier
$(1/2)\log n$. What matters is the size of the penalty gap between
higher-order optimal exact overfits and the minimal optimal order.
In particular, the theorem permits both different complexity scalings
across models and different sample-size scalings across $n$.

\begin{corollary}[Separable penalty systems]\label{cor:separable}
Under Assumption~\ref{ass:miss}, let 
\[
a_{k,n}=b_{n}c_{k},\qquad k\in[m],
\]
where $c_{k}>0$ is a fixed complexity score and $(b_{n})_{n\ge1}$
is a deterministic real sequence. Define 
\[
c_{*}=\min\{c_{k}:k\in K_{0}\}.
\]
If 
\[
b_{n}\to\infty,\qquad b_{n}=o(n),\qquad\frac{\log\log n}{b_{n}}\to0,
\]
and 
\[
c_{k}>c_{*}\qquad\text{for every }k\in K_{*}\text{ with }q_{k}>q_{*},
\]
then conditions \ref{it:P1}--\ref{it:P3} hold and therefore 
\[
\widehat{q}_{n}\to q_{*}\qquad\text{almost surely.}
\]
\end{corollary}
\begin{proof}
Condition \ref{it:P1} follows from $b_{n}=o(n)$. Also 
\[
a_{n}^{*}=b_{n}\min_{j\in K_{0}}c_{j}=b_{n}c_{*}.
\]
Hence, for every $k\in K_{*}$ with $q_{k}>q_{*}$, 
\[
a_{k,n}-a_{n}^{*}=b_{n}(c_{k}-c_{*}).
\]
Because $c_{k}-c_{*}>0$, the divergence $b_{n}\to\infty$ yields
\ref{it:P2}, and the condition $\log\log n=o(b_{n})$ yields \ref{it:P3}.
The conclusion is then immediate from Theorem~\ref{thm:main}. 
\end{proof}
Corollary~\ref{cor:separable} shows that the present theory covers
a rich family of BIC-type penalties. Ordinary BIC corresponds to $b_{n}=(1/2)\log n$
and $c_{k}=d_{k}$. Several penalties discussed by \citet{Huang2017}
translate to our maximization scale as follows. Bozdogan's consistent
AIC (CAIC) \citep{Bozdogan1987} uses 
\[
a_{k,n}=\frac{1}{2}d_{k}(\log n+1).
\]
Haughton's BIC$^{*}$ \citep{Haughton1988}, often called HBIC, uses
\[
a_{k,n}=\frac{1}{2}d_{k}\log\!\left(n/(2\pi)\right).
\]
Sclove's sample-size adjusted BIC \citep{Sclove1987} uses 
\[
a_{k,n}=\frac{1}{2}d_{k}\log\!\left((n+2)/24\right)
\]
for $n>22$. Each of these penalties is asymptotically equivalent
to the usual BIC penalty and therefore falls under Corollary~\ref{cor:separable}.
More generally, any BIC-type modification with multiplier $b_{n}\asymp\log n$
is covered, as are alternative logarithmic penalties such as $b_{n}=(\log n)^{\alpha}$
with any fixed $\alpha>0$ (cf. \citealp{Keribin2000}). By contrast, AIC's constant penalty is
excluded, and the classical Hannan--Quinn-type penalty \citep{HannanQuinn1979}
\[
a_{k,n}=d_{k}\log\log n
\]
lies exactly on the boundary of the present
argument rather than inside it.

\section{Discussion}

\label{sec:discussion}

It is useful to sharply contrast our result with neighboring strands
of the literature. The positive results discussed below are tailored
to particular asymptotic regimes, model classes, or inferential targets
rather than to the fixed-dimensional covariance-space setting treated
here. The present paper should be read against that background.

\subsection{Inapplicability of standard regularity arguments in factor models}

The most immediate comparison is with the regular misspecified-likelihood
literature. The model-selection results of \citet{Nishii1988} and
\citet{Vuong1989} are foundational, but they are designed for settings
in which pseudo-true parameters are locally regular enough for the
usual likelihood expansions. The same general pattern appears in \citet{SinWhite1996}
and, more recently, in the broader loss-based framework of \citet{Nguyen2024}:
one needs enough local smoothness
and curvature for likelihood or empirical-risk differences to admit
tractable second-order approximations around pseudo-true minimizers.
Factor analysis is not generically of that type. If the candidate
model adds redundant factors, the lower-order model is embedded in
the higher-order model along a singular stratum where extra loadings
are zero and rotations become unidentified. This is exactly the sort
of latent-variable irregularity highlighted by \citet{ChenMoustakiZhang2020}
in their analysis of likelihood-ratio tests. Their message is broader
than likelihood-ratio tests: once dimensionality comparisons in latent-variable
models violate Wilks regularity, BIC cannot be justified
by a routine appeal to regular asymptotics.

That observation also explains why one must be careful with blanket
statements of the form ``BIC is consistent for factor analysis.''
The answer depends on the asymptotic regime, the exact statistical
object being selected, and whether the argument is frequentist or
Bayesian. Our theorem answers only one of these settings: under a fixed observed
dimension, a fixed finite candidate family of compact structured Gaussian
covariance classes, unknown mean, and possible misspecification,
frequentist order selection by BIC penalization with
order-separating complexity counts is consistent.

\subsection{Local singular geometry behind the failure of regularity}

Although the proof never uses the loading parametrization directly,
a local heuristic in loading space helps explain why the regular assumptions
fail. Fix a pseudo-true covariance 
\[
\Sigma^{*}=\Lambda^{*}\Lambda^{*\top}+\Psi^{*}\in\mathcal{G}_{*}.
\]
Suppose an overfitted loading representation of order $q_{*}+r$ still
contains $\Sigma^{*}$ and admits redundant factor directions near
$\Sigma^{*}$. A representative local perturbation is 
\[
\Lambda(\theta)=\left[\Lambda^{*}+\theta A,\ \sqrt{\theta}\,B\right],\qquad\Psi(\theta)=\Psi^{*}+\theta\Delta,
\]
for small $\theta>0$. Then 
\[
\Lambda(\theta)\Lambda(\theta)^{\top}+\Psi(\theta)-\Sigma^{*}=\theta\left(\Lambda^{*}A^{\top}+A\Lambda^{*\top}+BB^{\top}+\Delta\right)+\theta^{2}AA^{\top}.
\]
Thus redundant factors appear at $\sqrt{\theta}$ scale in the loadings
but only at order $\theta$ in covariance space. This is the local
conic phenomenon familiar from singular model selection. To make the
algebra concrete in a dense model, suppose for the remainder of this
paragraph that $\Psi^{*}=\diag(\psi_{1}^{*},\dots,\psi_{p}^{*})$
is diagonal. The nonuniqueness can then be made completely explicit.
Take $A=0$, choose any nonzero scalar $b$, and let the redundant
block consist of a single extra column $B=be_{j}$, where $e_{j}$
is the $j$th coordinate vector in $\mathbb{R}^{p}$. Then $BB^{\top}=b^{2}e_{j}e_{j}^{\top}$
is diagonal. If we choose 
\[
\Delta=-b^{2}e_{j}e_{j}^{\top},
\]
then for every $\theta\in(0,\psi_{j}^{*}/b^{2})$ one has 
\[
\Lambda(\theta)=\left[\Lambda^{*},\sqrt{\theta}\,be_{j}\right],\qquad\Psi(\theta)=\Psi^{*}-\theta b^{2}e_{j}e_{j}^{\top}\succ0,
\]
and therefore 
\[
\Lambda(\theta)\Lambda(\theta)^{\top}+\Psi(\theta)=\Sigma^{*}.
\]
So the same covariance matrix is represented by a whole one-parameter
family of higher-order loading matrices, even within a diagonal-uniqueness
model. More complicated families arise from rotations and from redundant-factor
splittings. Unique pseudo-true parameters, nonsingular Hessians, and
ordinary quadratic likelihood expansions in loading coordinates can
therefore fail exactly where regular-model BIC arguments would need
them. Working directly in covariance space absorbs that singularity
into the geometry of the compact covariance classes, which is why
the present proof is organized there rather than in loading coordinates.

\subsection{Existing positive results in related factor settings}

\paragraph{Large-panel approximate factor models.}

The asymptotic regime of \citet{BaiNg2002}, \citet{AmengualWatson2007},
and \citet{ChoiJeong2019} is fundamentally different from ours. In
that literature, $N$ denotes the number of observed series and $T$
the number of time periods, and the consistency arguments require
both dimensions to diverge. \citet{BaiNg2002} treat static approximate
factor models, \citet{AmengualWatson2007} study the number of dynamic
factors in large panels, and \citet{ChoiJeong2019} derive several
likelihood-based information criteria for approximate factor models
in the same large-panel setting. The same caveat applies to closely
related contributions: \citet{HallinLiska2007} propose a spectral
information criterion for the number of dynamic shocks in the general
dynamic factor model; \citet{AlessiBarigozziCapasso2010} refine the
penalty of \citet{BaiNg2002} by tuning its multiplicative constant;
\citet{Onatski2010} use clustering of the largest idiosyncratic eigenvalues
near the upper edge of the sample spectrum; and \citet{AhnHorenstein2013}
select factor number by maximizing ratios of adjacent eigenvalues.
Critically, the model selection guarantees from these works do not
directly translate to the fixed-$p$ Gaussian covariance classes
of the form $\Lambda\Lambda^{\top}+\Psi$ with explicit structural
zero restrictions and BIC penalization studied here.

\paragraph{Structural equation modeling and confirmatory factor analysis.}

The asymptotic analysis in \citet{Huang2017} is directly relevant
because confirmatory factor analysis (CFA) is a special case of structural
equation modeling. That work studies AIC, BIC, and RMSEA under weak
distributional assumptions and defines the target through population
discrepancy minimization over a finite list of candidate models. It
further shows, among other things, that BIC can select the most parsimonious
model under nested selection but need not do so for general non-nested
ties. Our target is different: the minimal factor order inside explicit
covariance classes $\mathcal{F}_{k}$, with the proof designed to
survive singular loading geometry by operating directly in covariance
space. The theorem in \citet{Huang2017} is therefore a useful benchmark,
but it does not directly replace the covariance-space result proved here.

\paragraph{High-dimensional rank selection.}

The theorem of \citet{MorimotoHungHuang2026} unifies consistency
results for information-criterion-based rank estimators under random-matrix
asymptotics. That contribution is closely related in spirit, but it
concerns a different notion of model complexity via rank, defined
through high-dimensional spiked spectral behavior when both $p$ and
$n$ diverge and gap conditions hold. Our pseudo-true order is instead
defined through Gaussian Kullback--Leibler projection inside a fixed
family of structured covariance classes. The random-matrix theory
is therefore complementary rather than directly applicable in the
fixed-dimensional setting considered here.

\subsection{PanIC and generic empirical-risk penalties}

A natural comparison is with recent general theories for penalized
empirical-risk selection. \citet{SinWhite1996} prove consistency
results for penalized likelihood criteria in broad parametric settings,
and \citet{Baudry2015} relax part of that framework in a general
model-selection setting. More recently, \citet{Nguyen2024} develops
the PanIC framework, which gives general sufficient conditions for
consistent information criteria in loss-based model selection and
also provides broader sufficient conditions for BIC-like criteria
than were previously available. In a related minimum-empirical-risk
setting, \citet{WesterhoutEtAl2024} introduce the SWIC and show that
the usual AIC and BIC penalties do not satisfy their generic sufficient
condition.

These papers are relevant because they show that consistent selection
can often be recovered without exploiting model-specific geometry,
provided one uses penalties calibrated to the generic empirical-risk
fluctuation scale. In the minimum-empirical-risk normalization of
\citet{WesterhoutEtAl2024}, for instance, BIC contributes
a penalty of order $(\log n)/n$, whereas their SWIC penalty is of
order $(\log n)/\sqrt{n}$. The present theorem instead exploits the
special covariance geometry of structured factor analysis to show
that, once all globally optimal models share a common pseudo-true
covariance set, $Q$ has a local quadratic margin, and the BIC complexity
counts are order-separating at the pseudo-true order, the much lighter
BIC penalty already suffices.

\subsection{The AIC within our framework}

Within the context of our work, the AIC selects the model that maximizes
\[
T_{k,n}-d_{k}.
\]
Hence the penalty gap between a higher-order optimal model $k\in K_{*}$
and the minimal optimal class is simply 
\[
d_{k}-d_{*},
\]
which is constant in $n$. Consequently, whenever there exist optimal
exact overfits with $q_{k}>q_{*}$, conditions \ref{it:P2}--\ref{it:P3}
of Theorem~\ref{thm:main} fail for AIC. At the same time, Lemma~\ref{lem:exactoverfit}
gives only 
\[
T_{k,n}-U_{n}=O(\log\log n)\qquad\text{a.s.}
\]
for such exact overfits. A constant penalty therefore cannot dominate
the fluctuation scale that drives our BIC proof.

This is consistent with the familiar role of AIC in the broader model-selection
literature: AIC is designed around predictive efficiency rather than
consistent recovery of the smallest correct or pseudo-true model;
see \citet{Akaike1973}, \citet{Huang2017}, and \citet{ChoiJeong2019}.
In the present setting, one should likewise not expect a minimal-order
consistency theorem from AIC. We do not prove a separate overestimation
theorem here, but the rate comparison above shows why AIC falls outside
the mechanism that makes BIC consistent in Theorem~\ref{thm:main}.

\subsection{Regularization and penalized estimation}

Similar distinctions of aims appear in the regularization literature.
\citet{HiroseImada2018} study penalized maximum likelihood in factor
regression, with a lasso-type penalty on the loadings and an additional
penalty on diagonal error variances, to address three concrete estimation
problems: instability when the number of variables is large relative
to the sample size, insufficient sparsity from rotation-based maximum
likelihood estimation, and multicollinearity or improper solutions
caused by very small error-variance estimates. \citet{HiroseTerada2023}
propose the prenet penalty, whose aim is not merely sparsity but interpretable
simple structure; mild penalization approximates quartimin-type rotations,
while very large penalization yields perfect simple structure, and
AIC/BIC are then used to tune the regularization parameter or to exploit
zero columns in the loading matrix. These papers therefore focus on
penalized estimation and structural simplification rather than on the fixed-dimensional
order-consistency theorem proved here for BIC penalization
under the covariance-space assumptions and order-separating complexity
assignments adopted here.

\subsection{Singular learning theory}

The singular-model literature is perhaps the strongest reason to be
cautious about BIC in factor analysis. \citet{DrtonPlummer2017}
show that factor analysis is singular from the point of view of Bayesian
marginal likelihood and propose sBIC. \citet{Watanabe2013} develops
WBIC as a generally applicable approximation in singular models, and
\citet{DrtonEtAl2025} provide a recent singular learning-theoretic
analysis of factor analysis models. These contributions explain why
the Bayesian interpretation of classical BIC is more delicate in singular settings.

Our theorem does not contradict their results. Instead, it addresses
a different target. The quantities approximated by sBIC and WBIC are
singular Bayesian marginal likelihoods, and their penalties are driven
by learning coefficients rather than by Euclidean parameter counts
alone. The present theorem is frequentist. It says that, under explicit
covariance-space assumptions and an order-separating BIC complexity
assignment, the BIC penalty can still be strongly consistent
for recovering the smallest member among the Gaussian Kullback--Leibler-optimal
classes, even though that target need not coincide with the model
favored by singular Bayesian complexity.

\subsection{Limitations and extensions}

The theorem is intentionally narrow. It is a fixed-$p$, complete-data
result aimed at typical exploratory factor analysis applications.
It does not address missingness, incomplete-data likelihoods, dynamic
factor models, approximate factor panels with $N,T\to\infty$, or
high-dimensional $p/n\to c$ regimes. It also does not claim that
BIC is universally correct for singular Bayesian model selection.
Rather, it isolates the exact structural assumptions under which BIC penalization, together with an order-separating BIC complexity
assignment, still yields a strong frequentist order consistency theorem
for a fixed finite candidate family of compact covariance classes
in classical exploratory factor analysis.

The clearest next step is to study whether analogous common-projection
conditions can be verified or replaced in broader latent-variable
settings. Another is to understand the interface between frequentist
consistency and singular Bayesian penalties more precisely, especially
for sparse and confirmatory factor classes. A third is to extend the
present fixed-$p$ theory to incomplete-data settings, where AIC/BIC-type
procedures already exist computationally but where a theorem of the
present form would require new arguments. Finally, Section~\ref{subsec:penalties}
shows that the present argument actually validates a whole consistency
class of penalties, ranging from BIC to heavier sublinear
logarithmic penalties. At present, however, we do not have the sharper
local asymptotic tools needed to identify a canonical or minimal optimal
penalty within that class, or to describe the finite-sample trade-offs
among its members. In that sense, the theorem should be read as a
characterization of when consistency holds, rather than as a prescription
for which consistent penalty is best.

\appendix
%dummy comment inserted by tex2lyx to ensure that this paragraph is not empty

\section{Proofs of preliminary propositions}

\subsection{Proof of Proposition~\ref{prop:compact}}
\begin{proof}[Proof of Proposition~\ref{prop:compact}]
The loading class $\mathcal{L}_{k}$ is closed and bounded in the
Euclidean space $\mathbb{R}^{p\times q_{k}}$, hence compact. The
same is true of each error-covariance class $\mathcal{P}_{\tau_{k}}$.
The map 
\[
(\Lambda,\Psi)\mapsto\Lambda\Lambda^{\top}+\Psi
\]
is continuous, so $\mathcal{F}_{k}$ is compact as the continuous
image of a compact set. The finite union $\mathcal{K}$ is therefore
compact as well. For every $\Sigma=\Lambda\Lambda^{\top}+\Psi\in\mathcal{F}_{k}$
and every $v\in\mathbb{R}^{p}$, 
\[
v^{\top}\Sigma v=\|\Lambda^{\top}v\|^{2}+v^{\top}\Psi v\ge\underline{\psi}\|v\|^{2},
\]
which yields $\Sigma\succeq\underline{\psi}I_{p}$. Likewise, 
\[
\|\Sigma\|_{\mathrm{op}}\le\|\Lambda\Lambda^{\top}\|_{\mathrm{op}}+\|\Psi\|_{\mathrm{op}}\le\|\Lambda\|_{F}^{2}+\overline{\psi}\le M^{2}+\overline{\psi}.
\]
Thus $\Sigma\preceq(M^{2}+\overline{\psi})I_{p}$. In particular,
\[
\|\Sigma^{-1}\|_{\mathrm{op}}\le\underline{\psi}^{-1}\qquad\text{for every }\Sigma\in\mathcal{K},
\]
so the inverse map is uniformly bounded on $\mathcal{K}$. Finally,
for $\Sigma,\Gamma\in\mathcal{K}$, 
\[
\Sigma^{-1}-\Gamma^{-1}=\Sigma^{-1}(\Gamma-\Sigma)\Gamma^{-1},
\]
and therefore 
\[
\|\Sigma^{-1}-\Gamma^{-1}\|_{F}\le\|\Sigma^{-1}\|_{\mathrm{op}}\|\Gamma-\Sigma\|_{F}\|\Gamma^{-1}\|_{\mathrm{op}}\le\underline{\psi}^{-2}\|\Gamma-\Sigma\|_{F}.
\]
Hence the inverse map is Lipschitz on $\mathcal{K}$. 
\end{proof}

\subsection{Proof of Proposition~\ref{prop:profile}}
\begin{proof}[Proof of Proposition~\ref{prop:profile}]
Fix $\Sigma\in\mathbb{S}_{++}^{p}$. Expanding around the sample
mean gives 
\[
X_{t}-\mu=(X_{t}-\bar{X}_{n})+(\bar{X}_{n}-\mu),
\]
and therefore 
\begin{align*}
\sum_{t=1}^{n}(X_{t}-\mu)^{\top}\Sigma^{-1}(X_{t}-\mu) & =\sum_{t=1}^{n}(X_{t}-\bar{X}_{n})^{\top}\Sigma^{-1}(X_{t}-\bar{X}_{n})+n(\mu-\bar{X}_{n})^{\top}\Sigma^{-1}(\mu-\bar{X}_{n}),
\end{align*}
with the cross-term vanishing because $\sum_{t=1}^{n}(X_{t}-\bar{X}_{n})=0$.
Since $\Sigma^{-1}$ is positive definite, the second term is uniquely
minimized at $\mu=\bar{X}_{n}$. Substituting this maximizer into
the likelihood yields 
\[
\widetilde{\ell}_{n}(\Sigma)=-\frac{n}{2}\left\{\log\det\Sigma+\tr(S_{n}\Sigma^{-1})\right\}.
\]
The continuity of $\widetilde{\ell}_{n}$ on the compact set $\mathcal{F}_{k}$
then shows that $T_{k,n}$ is attained for every $k\in[m]$. 
\end{proof}

\section{Auxiliary lemmas and technical results}

\subsection{Population criterion and profiling over the mean}

If $\mathbb{E}\|X_{1}\|^{2}<\infty$ and 
\[
\ell_{1}(\mu,\Sigma)=-\frac{1}{2}\left\{\log\det\Sigma+(X_{1}-\mu)^{\top}\Sigma^{-1}(X_{1}-\mu)\right\}
\]
denotes the one-observation Gaussian log-likelihood up to the same
irrelevant additive constant, then 
\[
\Gamma(\mu,\Sigma)=\mathbb{E}[\ell_{1}(\mu,\Sigma)]=-\frac{1}{2}\left\{\log\det\Sigma+\tr(\Sigma_{0}\Sigma^{-1})+(\mu-\mu_{0})^{\top}\Sigma^{-1}(\mu-\mu_{0})\right\}.
\]
In particular, the quasi-likelihood risk $\Gamma(\mu,\Sigma)$ depends
on the data-generating law only through $(\mu_{0},\Sigma_{0})$. This
follows from the identity 
\[
\mathbb{E}[(X_{1}-\mu)(X_{1}-\mu)^{\top}]=\Sigma_{0}+(\mu-\mu_{0})(\mu-\mu_{0})^{\top},
\]
and the trace representation $x^{\top}Ax=\tr(Axx^{\top})$. Profiling
over $\mu$ immediately gives 
\[
Q(\Sigma)=\sup_{\mu\in\mathbb{R}^{p}}\Gamma(\mu,\Sigma)=-\frac{1}{2}\left\{\log\det\Sigma+\tr(\Sigma_{0}\Sigma^{-1})\right\}.
\]

\subsection{Empirical-process bounds}

\begin{lemma}[Iterated-logarithm bounds]\label{lem:lil} Assume
Assumption~\ref{it:M1}. Let 
\[
r_{n}=\sqrt{\frac{\log\log n}{n}}\qquad(n\ge3).
\]
Then 
\[
\|\bar{X}_{n}-\mu_{0}\|=O(r_{n})\quad\text{a.s.},\qquad\|S_{n}-\Sigma_{0}\|_{F}=O(r_{n})\quad\text{a.s.}
\]
\end{lemma}
\begin{proof}
Set $Y_{t}=X_{t}-\mu_{0}$. For each coordinate $j$, the scalar law
of the iterated logarithm gives 
\[
\frac{1}{n}\sum_{t=1}^{n}Y_{tj}=O(r_{n})\qquad\text{a.s.}
\]
Because $p$ is fixed, taking the maximum over finitely many coordinates
yields the vector bound for $\bar{X}_{n}-\mu_{0}$. For the covariance
term, write 
\[
\widetilde{S}_{n}=\frac{1}{n}\sum_{t=1}^{n}Y_{t}Y_{t}^{\top}.
\]
Each entry of $\widetilde{S}_{n}-\Sigma_{0}$ is an average of centered
iid variables of the form $Y_{tj}Y_{t\ell}-\mathbb{E}[Y_{1j}Y_{1\ell}]$.
Assumption~\ref{it:M1} gives $\mathbb{E}\|X_{1}\|^{4}<\infty$,
hence these centered products have finite second moment. The scalar
law of the iterated logarithm therefore yields entrywise $O(r_{n})$
almost surely. Since 
\[
S_{n}=\widetilde{S}_{n}-(\bar{X}_{n}-\mu_{0})(\bar{X}_{n}-\mu_{0})^{\top},
\]
the Frobenius norm of $S_{n}-\Sigma_{0}$ is bounded by the corresponding
norm of $\widetilde{S}_{n}-\Sigma_{0}$ plus $\|\bar{X}_{n}-\mu_{0}\|^{2}=O(r_{n}^{2})$,
which is negligible relative to $r_{n}$. 
\end{proof}
\begin{lemma}[Uniform approximation on compact covariance sets]\label{lem:uniform}
Let $\mathcal{S}\subset\mathbb{S}_{++}^{p}$ be compact. Then there
exists a finite constant $C_{\mathcal{S}}$ such that 
\[
\sup_{\Sigma\in\mathcal{S}}\left|\widetilde{\ell}_{n}(\Sigma)-nQ(\Sigma)\right|\le C_{\mathcal{S}}n\|S_{n}-\Sigma_{0}\|_{F}.
\]
Under Assumption~\ref{it:M1}, 
\[
\sup_{\Sigma\in\mathcal{S}}\left|\widetilde{\ell}_{n}(\Sigma)-nQ(\Sigma)\right|=O\left(\sqrt{n\log\log n}\right)\qquad\text{a.s.}
\]
In particular, 
\[
\sup_{\Sigma\in\mathcal{S}}\left|\frac{1}{n}\widetilde{\ell}_{n}(\Sigma)-Q(\Sigma)\right|\to0\qquad\text{a.s.}
\]
\end{lemma}
\begin{proof}
Let $\Delta_{n}=S_{n}-\Sigma_{0}$. Using the formulas for $\widetilde{\ell}_{n}$
and $Q$, 
\[
\widetilde{\ell}_{n}(\Sigma)-nQ(\Sigma)=-\frac{n}{2}\tr(\Delta_{n}\Sigma^{-1}).
\]
Hence, by Cauchy--Schwarz in Frobenius norm, 
\[
\left|\widetilde{\ell}_{n}(\Sigma)-nQ(\Sigma)\right|\le\frac{n}{2}\|\Delta_{n}\|_{F}\,\|\Sigma^{-1}\|_{F}.
\]
Since $\Sigma\mapsto\|\Sigma^{-1}\|_{F}$ is continuous on the compact
set $\mathcal{S}$, it is bounded there. The stated inequalities follow
immediately from Lemma~\ref{lem:lil}. 
\end{proof}

\subsection{Likelihood comparisons}

\begin{lemma}[Suboptimal models lose likelihood linearly in $n$]\label{lem:suboptimal}
If $k\notin K_{*}$, then 
\[
T_{k,n}=nV_{k}+o(n)\qquad\text{a.s.},\qquad T_{k,n}-U_{n}=-(V_{*}-V_{k})n+o(n)\qquad\text{a.s.}
\]
\end{lemma}
\begin{proof}
Apply Lemma~\ref{lem:uniform} with $\mathcal{S}=\mathcal{F}_{k}$
and again with $\mathcal{S}=\mathcal{G}_{*}$. Since $Q$ is continuous
and the relevant sets are compact, 
\[
T_{k,n}=nV_{k}+o(n),\qquad U_{n}=nV_{*}+o(n)\qquad\text{a.s.}
\]
Subtracting the two displays yields the result. 
\end{proof}
\begin{lemma}[Likelihood gain from overfitting]\label{lem:exactoverfit}
Take Assumption~\ref{ass:miss} and fix $k\in K_{*}$. Then there
exists an almost surely finite random constant $C_{k}$ such that
\[
T_{k,n}\le U_{n}+C_{k}\log\log n
\]
for all sufficiently large $n$, almost surely. Since $\mathcal{G}_{*}\subset\mathcal{F}_{k}$,
one also has $T_{k,n}\ge U_{n}$, and thus 
\[
T_{k,n}-U_{n}=O(\log\log n)\qquad\text{a.s.}
\]
\end{lemma}
\begin{proof}
Fix $k\in K_{*}$. For $\Sigma\in\mathcal{F}_{k}$, write $d(\Sigma)=\dist(\Sigma,\mathcal{G}_{*})$
and, because $\mathcal{G}_{*}$ is compact, choose a projection point
$\Pi(\Sigma)\in\mathcal{G}_{*}$ with $\|\Sigma-\Pi(\Sigma)\|_{F}=d(\Sigma)$.
Let $\Delta_{n}=S_{n}-\Sigma_{0}$. Since $U_{n}\ge\widetilde{\ell}_{n}(\Pi(\Sigma))$, 
\[
\widetilde{\ell}_{n}(\Sigma)-U_{n}\le\widetilde{\ell}_{n}(\Sigma)-\widetilde{\ell}_{n}(\Pi(\Sigma)).
\]
Using the identity from Lemma~\ref{lem:uniform}, 
\[
\widetilde{\ell}_{n}(\Sigma)-\widetilde{\ell}_{n}(\Pi(\Sigma))=n\left(Q(\Sigma)-V_{*}\right)-\frac{n}{2}\tr\left(\Delta_{n}(\Sigma^{-1}-\Pi(\Sigma)^{-1})\right).
\]
The inverse map is Lipschitz on $\mathcal{K}$, so 
\[
\|\Sigma^{-1}-\Pi(\Sigma)^{-1}\|_{F}\le\underline{\psi}^{-2}d(\Sigma).
\]
Hence 
\[
\widetilde{\ell}_{n}(\Sigma)-U_{n}\le n\left(Q(\Sigma)-V_{*}\right)+A_{1}\sqrt{n\log\log n}\,d(\Sigma)
\]
eventually almost surely, for some almost surely finite random constant
$A_{1}$.

Now split $\mathcal{F}_{k}$ into a near region $\{d(\Sigma)<\eta_{k}\}$
and a far region $\{d(\Sigma)\ge\eta_{k}\}$. On the near region,
Assumption~\ref{it:M3} implies 
\[
Q(\Sigma)-V_{*}\le-c_{k}d(\Sigma)^{2},
\]
and therefore 
\begin{align*}
\widetilde{\ell}_{n}(\Sigma)-U_{n} & \le-nc_{k}d(\Sigma)^{2}+A_{1}\sqrt{n\log\log n}\,d(\Sigma)\\
 & =-nc_{k}\left(d(\Sigma)-\frac{A_{1}}{2c_{k}}\sqrt{\frac{\log\log n}{n}}\right)^{2}+\frac{A_{1}^{2}}{4c_{k}}\log\log n\\
 & \le\frac{A_{1}^{2}}{4c_{k}}\log\log n.
\end{align*}
On the far region, set 
\[
\mathcal{F}_{k}^{\mathrm{far}}=\{\Sigma\in\mathcal{F}_{k}:d(\Sigma)\ge\eta_{k}\}.
\]
If $\mathcal{F}_{k}^{\mathrm{far}}=\varnothing$, then the near-region
bound already proves the claim. Otherwise, because $d(\cdot)$
is continuous and $\mathcal{F}_{k}$ is compact, the set $\mathcal{F}_{k}^{\mathrm{far}}$
is compact. If $\sup_{\Sigma\in\mathcal{F}_{k}^{\mathrm{far}}}Q(\Sigma)=V_{*}$,
then continuity of $Q$ on $\mathcal{F}_{k}^{\mathrm{far}}$ would
yield some $\Sigma^{\dagger}\in\mathcal{F}_{k}^{\mathrm{far}}$ with
$Q(\Sigma^{\dagger})=V_{*}$. But then $\Sigma^{\dagger}\in\mathcal{G}_{k}=\mathcal{G}_{*}$
by Assumption~\ref{it:M2}, which contradicts $d(\Sigma^{\dagger})\ge\eta_{k}>0$.
Therefore 
\[
b_{k}=V_{*}-\sup_{\Sigma\in\mathcal{F}_{k}^{\mathrm{far}}}Q(\Sigma)>0,
\]
so $Q(\Sigma)\le V_{*}-b_{k}$ whenever $d(\Sigma)\ge\eta_{k}$. Hence
\[
\widetilde{\ell}_{n}(\Sigma)-U_{n}\le-b_{k}n+A_{1}D_{*}\sqrt{n\log\log n},
\]
where $D_{*}=\sup\{\|\Sigma-\Gamma\|_{F}:\Sigma,\Gamma\in\mathcal{K}\}<\infty$.
Because $\sqrt{n\log\log n}=o(n)$, the right-hand side is negative
for all sufficiently large $n$, almost surely. Taking suprema over
the near and far regions shows that 
\[
T_{k,n}\le U_{n}+\frac{A_{1}^{2}}{4c_{k}}\log\log n
\]
eventually almost surely. The lower bound $T_{k,n}\ge U_{n}$ is immediate
from $\mathcal{G}_{*}\subset\mathcal{F}_{k}$. 
\end{proof}

\section{Proofs of the main results}

\subsection{Proof of Theorem~\ref{thm:main}}
\begin{proof}[Proof of Theorem~\ref{thm:main}]
Define the common benchmark 
\[
B_{n}=U_{n}-a_{n}^{*}.
\]
Because every model in $K_{0}$ contains $\mathcal{G}_{*}$, one has
\[
\max_{j\in K_{0}}W_{j,n}\ge B_{n}.
\]
It is therefore enough to prove that every model outside $K_{0}$
eventually has score strictly below $B_{n}$.

If $k\notin K_{*}$, then Lemma~\ref{lem:suboptimal} gives 
\[
T_{k,n}-U_{n}=-(V_{*}-V_{k})n+o(n)\qquad\text{a.s.}
\]
Since $a_{k,n}=o(n)$ by \ref{it:P1} and $a_{n}^{*}=o(n)$ as well,
\[
W_{k,n}-B_{n}=(T_{k,n}-U_{n})-(a_{k,n}-a_{n}^{*})\to-\infty\qquad\text{a.s.}
\]
Thus every suboptimal model is eventually rejected.

Now fix $k\in K_{*}$ with $q_{k}>q_{*}$. Lemma~\ref{lem:exactoverfit}
gives 
\[
T_{k,n}-U_{n}=O(\log\log n)\qquad\text{a.s.}
\]
Therefore 
\[
W_{k,n}-B_{n}=(T_{k,n}-U_{n})-(a_{k,n}-a_{n}^{*})\le C_{k}\log\log n-(a_{k,n}-a_{n}^{*})
\]
eventually almost surely for some almost surely finite random constant
$C_{k}$. By \ref{it:P2}--\ref{it:P3}, the penalty difference dominates
$\log\log n$, so again $W_{k,n}-B_{n}\to-\infty$ almost surely.
Hence no globally optimal model of strictly larger order can be selected
eventually.

We have therefore shown, on an event of probability one, that for
every $k\notin K_{0}$ there exists a finite random index $N_{k}$
such that $W_{k,n}<B_{n}$ for all $n\ge N_{k}$. If $K_{0}=[m]$, then
all candidate models already have order $q_{*}$ and there is nothing
more to prove. Otherwise, because the candidate family is finite, the
maximum 
\[
N=\max\{N_{k}:k\notin K_{0}\}
\]
is finite on that event. For every $n\ge N$, every maximizer of the
penalized scores must therefore belong to $K_{0}$. Since every model
in $K_{0}$ has order exactly $q_{*}$, it follows that 
\[
\widehat{q}_{n}=q_{*}\qquad\text{for all sufficiently large }n\quad\text{a.s.}
\]
This proves strong consistency. 
\end{proof}

\subsection{Proof of Corollary~\ref{cor:bic}}
\begin{proof}[Proof of Corollary~\ref{cor:bic}]
For BIC, take 
\[
a_{k,n}=\frac{1}{2}d_{k}\log n.
\]
Then \ref{it:P1} holds because $\log n=o(n)$. Moreover, 
\[
a_{n}^{*}=\min_{j\in K_{0}}\frac{1}{2}d_{j}\log n=\frac{1}{2}d_{*}\log n.
\]
If $k\in K_{*}$ and $q_{k}>q_{*}$, then by the order-separation
assumption, 
\[
a_{k,n}-a_{n}^{*}=\frac{1}{2}(d_{k}-d_{*})\log n\to\infty,
\]
which is \ref{it:P2}. Also, 
\[
\frac{\log\log n}{a_{k,n}-a_{n}^{*}}=\frac{2\log\log n}{(d_{k}-d_{*})\log n}\to0,
\]
which is \ref{it:P3}. Theorem~\ref{thm:main} therefore yields the
order consistency. The model-class conclusion follows by repeating
the same comparison for models in $K_{0}\setminus K_{**}$, whose
penalty gap relative to $d_{*}$ is again proportional to $\log n$.

\end{proof}

\end{document}